\documentclass[a4paper, 12pt]{article}
\usepackage{amsfonts}
\usepackage {amssymb}
\usepackage {amsmath}
\usepackage {amsmath}
\usepackage {amsthm}
\usepackage{graphicx}
\usepackage{multirow}
\usepackage{xypic}
\usepackage {amscd}
\usepackage{mathrsfs}
\usepackage[colorlinks, linkcolor=blue, anchorcolor=black, citecolor=red]{hyperref}
\usepackage{enumerate}

\usepackage{geometry}  

\newcommand{\mcY}{{\mathcal{Y}}}
\newcommand{\mcL}{{\mathcal{L}}}
\newcommand{\mcX}{{\mathcal{X}}}

\newcommand{\vol}{{\rm vol}}
\newcommand{\ord}{{\rm ord}}

\newcommand{\vphi}{\varphi}

\newcommand{\bC}{{\mathbb{C}}}

\newcommand{\bP}{{\mathbb{P}}}

\newcommand{\fb}{\mathfrak{b}}

\newcommand{\NA}{{\rm NA}}
\newcommand{\MA}{{\rm MA}}

\newcommand{\triv}{{\rm triv}}
\newcommand{\bQ}{{\mathbb{Q}}}

\newcommand{\bR}{{\mathbb{R}}}

\newcommand{\mcH}{{\mathcal{H}}}
\newcommand{\mcO}{{\mathcal{O}}}

\newcommand{\bZ}{{\mathbb{Z}}}

\newcommand{\red}{{\rm red}}

\newcommand{\supp}{{\rm Supp}}

\newcommand{\bN}{{\mathbb{N}}}

\newcommand{\la}{\langle}
\newcommand{\ra}{\rangle}

\newcommand{\bfM}{{\bf M}}
\newcommand{\bfH}{{\bf H}}
\newcommand{\bfE}{{\bf E}}
\newcommand{\bfJ}{{\bf J}}

\newcommand{\Aut}{{\rm Aut}}
\newcommand{\PSH}{{\rm PSH}}

\newcommand{\ud}{\underline}

\newcommand{\cO}{\mathcal{O}}

\newcommand{\bfR}{{\bf R}}

\newcommand{\bfB}{{\bf B}}

\newcommand{\cJ}{\mathcal{J}}

\newcommand{\ddc}{{\rm dd^c}}

\newcommand{\mcP}{\mathcal{P}}
\newcommand{\mcN}{\mathcal{N}}
\newcommand{\fr}{\mathfrak{r}}

\newcommand{\bbL}{\mathbb{L}}

\newtheorem{thm}{Theorem}[section]
\newtheorem{prop}[thm]{Proposition}
\newtheorem{defn}[thm]{Definition}

\newtheorem{rem}[thm]{Remark}
\newtheorem{conj}[thm]{Conjecture}

\newtheorem{lem}[thm]{Lemma}

\newtheorem{defn-prop}[thm]{Definition-Proposition}

\begin{document}

\title{K-stability and Fujita approximation}
\author{Chi Li}
\date{}

\maketitle


\abstract{
This note is a continuation to the paper \cite{Li20}. 
We derive a formula for non-Archimedean Monge-Amp\`{e}re measures of big models. As applications, we derive a positive intersection formula for non-Archimedean Mabuchi functional, and further reduces the ($\Aut(X, L)_0$)-uniform Yau-Tian-Donaldson conjecture for polarized manifolds to a conjecture on the 
existence of approximate Zariski decompositions that satisfy some asymptotic vanishing condition. In an appendix, we also verify this conjecture for some of Nakayama's examples that do not admit birational Zariski decompositions. }

\setcounter{tocdepth}{1}
\tableofcontents


\section{Introduction}

Let $(X, L)$ be a polarized projective manifold. The Yau-Tian-Donaldson (YTD) conjecture predicts that the existence of constant scalar curvature K\"{a}hler (cscK) metrics in the K\"{a}hler class $c_1(L)$ is equivalent to a K-stability condition for the pair $(X, L)$. The K-stability condition is usually expressed as a positivity condition on the Futaki invariants of test configurations. In a recent work \cite{Li20}, it was proved that the existence of cscK metrics is equivalent to the uniform positivity of Mabuchi slopes along all maximal geodesic rays. 
Here the maximal geodesic rays, as introduced by Berman-Boucksom-Jonsson \cite{BBJ18}, are essentially the geodesic rays in the space of (mildly singular) positive metrics in $c_1(L)$ that can be algebraically approximated by the data of test configurations. It is known that for test configurations, the Mabuchi slopes (of geodesic rays associated to test configurations) are the Futaki invariants. 
So our result is of a Yau-Tian-Donaldson type. However the approximability of Mabuchi slopes of (maximal) geodesic rays, is not well-understood yet. In \cite{Li20}, we did a partial comparison between the Mabuchi slope with non-Archimedean Mabuchi functional
and reduced the ($G$-)uniform version of YTD conjecture to a non-Archimedean version of entropy regularization conjecture of Boucksom-Jonsson (\cite{BoJ18b}). 

Furthermore we carried out a partial regularization process (based on Boucksom-Favre-Jonsson's work on Non-Achimedean Calabi-Yau theorems) and proved that K-stability for model filtrations is a sufficient (and conjecturally also a necessary) condition for the existence of cscK metrics. By a model filtration, we mean a filtration of the section ring $R(X, L)=\bigoplus_{m=0}^{+\infty} H^0(X, mL)$ induced by a model $(\mcX, \mcL)$ of $(X, L)$. See Definition \ref{def-model} for the definition of a model, for which 
the $\bQ$-line bundle $\mcL$ is not assumed to be semiample compared to a test configuration in the usual definition of K-stability (see \cite{Tia97, Don02}).
The main goal of this paper is to further reduce Boucksom-Jonsson's non-Archimedean regularization conjecture and hence the YTD conjecture to some purely algebro-geometric conjecture about big line bundles (see Conjecture \ref{conj-Fujmodel}, or more generally Conjecture \ref{conj-Fuj}, for the conjectural statements), which could be studied even without the background on K-stability or non-Archimedean geometry.

More specifically, we will first derive a formula for the non-Archimedean Monge-Amp\`{e}re measure of big models, which implies a positive intersection formula for the non-Archimedean Mabuchi functional of model filtrations. We refer to section \ref{sec-prel} for definitions of terms in the following statement of our main results.

\begin{thm}\label{thm-main}
For any normal and {\it big} model $(\mcX, \mcL)$ of $(X, L)$, if $\phi_{(\mcX, \mcL)}$ denotes the associated non-Archimedean psh metric,
then the following statements hold true.
\begin{enumerate}[(i)]
\item
If the central fibre is given by $\mcX_0=\sum_{i=1}^I b_i E_i$, and $x_i=r(b_i^{-1}\ord_{E_i})$ is the Shilov point associated to $E_i$, then
the non-Archimedean Monge-Amp\`{e}re measure of $\phi_{(\mcX, \mcL)}$ is given by the formula:
\begin{equation}\label{eq-MAmodel}
\MA^\NA(\phi_{(\mcX, \mcL)})=\sum_{i=1}^I b_i \left(\la \bar{\mcL}^n\ra\cdot E_i\right) \delta_{x_i},
\end{equation}
where $\la \bar{\mcL}^n\ra\in H^{n, n}(\bar{\mcX})$ is the positive intersection product of big line bundles (see section \ref{sec-resvol}). 
\item
The non-Archimedean Mabuchi functional of any big model $(\mcX, \mcL)$ is given by:
\begin{equation}\label{eq-Mmodel}
\bfM^\NA(\mcX, \mcL)=\la \bar{\mcL}^n\ra \cdot \left(K^{\log}_{\bar{\mcX}/\bP^1}+\frac{\ud{S}}{n+1}\bar{\mcL}\right).
\end{equation}
\end{enumerate}

\end{thm}

The non-Archimedean Monge-Amp\`{e}re measure on Berkovich spaces were introduced by A. Chambert-Loir \cite{CL06} and 
the formula \eqref{eq-MAmodel} is a generalization of the formula of non-Archimedean Monge-Amp\`{e}re measures for smooth semipositive non-Archimedean metrics.
We refer to section \ref{sec-resvol} for the definition of positive intersection numbers that arise in the study of restricted volumes of big line bundles.

The formula \eqref{eq-Mmodel}, which was announced in \cite{Li20}, generalizes the intersection formula for non-Archimedean Mabuchi functional of a test configuration (\cite{BHJ19, Tia17}) which coincides with the CM weight when the central fibre of the test configuration is reduced (see \cite{Wan12, Oda13, LX14}). As mentioned above, it together with the work in \cite{Li20} further reduce the proof of YTD conjecture to some algebraic conjecture (Conjecture \ref{conj-Fujmodel}).
Here for the convenience of the reader we recall the main result from \cite{Li20}, which is the recent progress in the variational approach to the YTD conjecture (as proposed in \cite{BBJ18, Bo18}) and incorporates the analytic existence result of Chen-Cheng \cite{CC18}.
\begin{defn}\label{def-uniK}
$(X, L)$ is uniformly K-stable for models if there exists $\gamma>0$ such that for any model $(\mcX, \mcL)$, we have:
\begin{equation}
\bfM^\NA(\mcX, \mcL)\ge \gamma \cdot \bfJ^\NA(\mcX, \mcL)
\end{equation}
where $\bfM^\NA$ and $\bfJ^\NA$ are given in \eqref{eq-MNA}-\eqref{eq-JNA}. 
\end{defn}
\begin{thm}[\cite{Li20}]\label{thm-YTD}
If a polarized manifold $(X, L)$ is uniformly K-stable for models, then $(X, L)$ admits a cscK metric. 
\end{thm}
We will see that
the positive intersection formula \eqref{eq-Mmodel} implies that it suffices to test the uniform K-stability for the models with reduced central fibres in which case $K^{\log}_{\bar{\mcX}/\bP^1}=K_{\mcX/\bP^1}$ (see Proposition \ref{prop-reduced}). 

The converse direction of Theorem \ref{thm-YTD} is expected to be true if $\Aut(X, L)_0$ is discrete. Indeed, it is implied by Conjecture \ref{conj-Fujmodel}. Moreover there is a version in the case when $\Aut(X, L)_0$ is not discrete (see \cite{Li20} for details). 
As observed by Y. Odaka, such results can be applied to get immediately the G-uniform version of Yau-Tian-Donaldson conjecture for polarized spherical manifolds (see some beautiful refinement by T. Delcroix \cite{Del20, Del20b} in this case and Remark \ref{rem-spherical}).


We end this introduction with the organization of this paper. In section \ref{sec-model}, we recall the construction of non-Archimedean psh metrics from models. In section \ref{sec-resvol} we recall the concepts related to restricted volumes of big line bundles and positive intersection products, and important results from \cite{BFJ09, ELMNP09} about the relation between them. In section \ref{sec-mov}, we prove Theorem \ref{thm-main}. 
 In the section \ref{sec-conj}, we propose a general conjecture which strengthens the usual Fujita approximation theorem and (in the $\bC^*$-equivariant case) would imply the uniform YTD conjecture for cscK metrics. In the appendix, we verify this algebraic conjecture for some of Nakayama's examples that do not admits birational Zariski decompositions. 
\vskip 1mm
\vskip 1mm
{\bf Acknowledgement:} 
The author is partially supported by NSF (Grant No. DMS-1810867) and an Alfred P. Sloan research fellowship. Part of this paper was written when the author was working at Purdue University. I would like to thank members of algebraic geometry group at Purdue, especially Sai-Kee Yeung, for their interests in this work, and Linquan Ma for helpful discussions. I am grateful to Sebastien Boucksom and Mattias Jonsson for their interest in this work and patient discussions about Conjecture \ref{conj-Fuj}, especially for suggesting ways to approach it and pointing out some delicate difficulty (see Remark \ref{rem-Boucksom}). 


\section{Preliminaries}\label{sec-prel}
\subsection{Non-Archimedean metrics associated to models}\label{sec-model}

This paper is a following-up work of \cite{Li20} and we will mostly follow the notations from that work. 
\begin{defn}\label{def-model}
\begin{itemize}
\item
A model of $(X, L)$ is a flat family of projective varieties $\pi: \mcX\rightarrow \bC$ together with a $\bQ$-line bundle $\mcL$ satisfying:
\begin{enumerate}
\item[(i)] There is a $\bC^*$-action on $(\mcX, \mcL)$ such that $\pi$ is $\bC^*$-equivariant;
\item[(ii)] There is a $\bC^*$-equivariant isomorphism $(\mcX, \mcL)\times_{\bC}\bC^* \cong (X, L)\times\bC^*$.
\end{enumerate}
\item
The trivial model of $(X, L)$ is given by $(X\times\bC, L\times\bC)=:(X_\bC, L_\bC)$.

Two models $(\mcX_i, \mcL_i), i=1,2$ are called equivalent if there exists a model $(\mcX_3, \mcL_3)$ and two $\bC^*$-equivariant birational morphisms $\mu_i: \mcX_3\rightarrow \mcX_i$ such that $\mu_1^*\mcL_1=\mu_2^*\mcL_2$.

\item
If we forget about the data $L$ and $\mcL$, then we say that $\mcX$ is a model of $X$.

If there is a $\bC^*$-equivariant birational morphism $r_{\mcX_1, \mcX_2}: \mcX_1\rightarrow \mcX_2$ for two models $\mcX_i, i=1,2$, then we say that $\mcX_1$ dominates $\mcX_2$ and write $\mcX_1\ge \mcX_2$. If $\mcX\ge X_\bC$, then we say that $\mcX$ is dominating.

If $\mcX$ is normal, we say that $\mcX$ is a normal model. 
We say a model $\mcX$ is a SNC (i.e. simple normal crossing) if $(\mcX, \mcX^\red_0)$ is a simple normal crossing pair. 


\item
Let $(\bar{\mcX}, \bar{\mcL})$ be the canonical $\bC^*$-equivariant compactification of $(\mcX, \mcL)$ over $\bP^1$ by adding the trivial $(X, L)$ at $\infty\in \bP^1$. 

We say that $(\mcX, \mcL)$ is a big model if $\bar{\mcL}$ is a big $\bQ$-line bundle over $\bar{\mcX}$ and the stable base ideal of $m\bar{\mcL}$ is the same as the $\pi$-base ideal of $m\mcL$ for $m\gg 1$. In particular, the stable base locus satisfies $\bfB(\mcL)\subseteq \mcX_0=\pi^{-1}(\{0\})$. (This definition is motivated by \cite[Lemma A.6]{BoJ18b}.)

In the following for simplicity of notations, if there is no confusion, we also just write $(\mcX, \mcL)$ for $(\bar{\mcX}, \bar{\mcL})$.

\item
If $\mcL$ is semiample over $\bC$, then we call the model $(\mcX, \mcL)$ to be a test configuration of $(X, L)$.
\end{itemize}
\end{defn}

\begin{rem}
Rigorously speaking, the model of $(X, L)$ should be called the model of $(X\times\bC, L\times\bC)$. In other words, with the language of \cite{BoJ18a}, we used the base change from the trivially valued case to the discrete valued case. 

In the original literature of K-stability, which we adopt in this paper, the line bundle $\mcL$ is assumed to be semi-ample. For us this is the only difference between the definition of test configurations and models. \end{rem}

We refer to \cite{BFJ15, BoJ18a} for the definition of Berkovich analytification $(X^\NA, L^\NA)$ of $(X, L)$ with respect to the trivially valued field $\bC$ and the definition of non-Archimedean psh metrics $L^\NA$ which are represented by $\phi_\triv$-psh functions on $X^\NA$ (where
$\phi_\triv$ is the metric associated to the trivial test configuration).

For each model $(\mcX, \mcL)$ of $(X, L)$, we can associate a non-Archimedean psh metric $\phi_{(\mcX, \mcL)}$ in the following way.  
If $\fb_m$ denotes the $\pi$-relative base ideal of $m\mcL$ and
$\mu_m: \mcX_m\rightarrow \mcX$ is the normalized blowup of $\fb_m$ with the exceptional divisor denoted by $\tilde{E}_m$, then 
$(\mcX_m, \mcL_m=\mu_m^*\mcL-\frac{1}{m}\tilde{E}_m)$ is a semiample test configuration. $(\mcX_m, \mcL_m)$ defines a smooth non-Archimedean metric $\phi_{(\mcX_m, \mcL_m)}\in \mcH^\NA(L)$ and we set
\begin{equation}
\phi_{(\mcX, \mcL)}=\lim_{m\rightarrow+\infty} \phi_{(\mcX_m, \mcL_m)}.
\end{equation} 
If the base variety $\mcX$ is clear, we just write $\phi_{(\mcX, \mcL)}$ as $\phi_{\mcL}$. It is easy to see that equivalent models define the same non-Archimedean psh metrics. Moreover, if $\mcL$ is semiample, then $\phi_{(\mcX, \mcL)}=\phi_{(\mcX_m, \mcL_m)}$ for $m$ sufficiently divisible.
 
By resolution of singularities, we can assume that $\mcX$ is dominating via a $\bC^*$-equivariant birational morphism $\rho: \mcX\rightarrow X_\bC$. Write $\mcL=\rho^*L+D$ with $D$ supported on $\mcX_0$. Then $\mcL$ defines a model function $f_\mcL$ on $X^{\rm div}_\bQ$ (the set of divisorial valuations on $X$) given by:
\begin{equation}
f_\mcL(v)=G(v)(D), \quad \forall v\in X^{\rm div}_\bQ
\end{equation}
where $G(v): X^{\rm div}_\bQ\rightarrow (X\times\bC)^{\rm div}_\bQ$ is the Gauss extension, i.e. $G(v)$ is a $\bC^*$-invariant valuation on $X\times\bC$ that extends $v$ and satisfies $G(v)(t)=1$. Set $\tilde{\phi}_\mcL=\phi_{\triv}+f_\mcL$. 
 
The $\phi_\triv$-psh upper envelope of $f_\mcL$ is defined as:
\begin{equation}
P(f_\mcL)(v)=\sup\left\{(\phi-\phi_{\triv})(v); \phi\in \PSH^\NA(L), \phi-\phi_\triv\le f_\mcL \right\}.
\end{equation}
By \cite[Theorem 8.5]{BFJ16} we have the identity $\phi_{(\mcX, \mcL)}=\phi_\triv+P(f_\mcL)=:P(\tilde{\phi}_\mcL)$.
Moreover, by \cite[Theorem 8.3]{BFJ16}, $P(f_\mcL)$ is a continuous $\phi_{\triv}$-psh function. 

Because $\bar{\mcL}$ is $\bar{\pi}$-big 
over the compactification $\bar{\mcX}\stackrel{\bar{\pi}}{\rightarrow} \bP^1$, when $c\gg 1$, the $\bQ$-line bundle 
$\bar{\mcL}_c:=\bar{\mcL}+c \mcX_0$ is big over $\bar{\mcX}$. Moreover, by \cite[Lemma A.8]{BFJ15}, when $c\gg 1$,  the $\pi$-relative base ideal of $m \bar{\mcL}$ is the same as the absolute base ideal of $m \bar{\mcL}$ for all $m$ sufficiently divisible. In other words we know that $(\mcX, \mcL_c)$ is a big model in the sense in Definition \ref{def-model}. 
Note that we have $P(\tilde{\phi}_\mcL)+c=P(\tilde{\phi}_{\mcL_c})=\phi_\triv+P(f_{\mcL_c})$.


\subsection{Restricted volumes and positive intersection products}\label{sec-resvol}

In this section, we (change the notation and) assume that $\mcX$ is a compact projective manifold and $\mcL$ is a big line bundle over $\mcX$ of dimensional $n+1$. Recall that the volume of $\mcL$ is defined as:
\begin{equation}\label{eq-vol}
\vol_\mcX(\mcL)=\limsup_{m\rightarrow+\infty} \frac{h^0(\mcX, m\mcL)}{m^{n+1}/(n+1)!}.
\end{equation}
Denote by $N^1(\mcX)={\rm Div}(\mcX)/\equiv\;$ the N\'{e}ron-Severi group. Then the volume functional extends to be a continuous function on $N^1(\mcX)_\bR=N^1(\mcX)\otimes_\bQ \bR$. By Fujita's approximation theorem, this invariant can be calculated as the movable intersection number of $\mcL$ (see \cite{DEL00, Laz04}). In other words, if we let $\mu_m: \mcX_m\rightarrow \mcX$ be the normalized blowup of $\fb(|mL|)$ (or its resolution) with exceptional divisor $\tilde{E}_m$ and set $\mcL_m=\mu_m^*\mcL-\frac{1}{m}\tilde{E}_m$, then
\begin{equation}
\vol_{\mcX}(\mcL)=\lim_{m\rightarrow+\infty}\mcL_m^{n+1}.
\end{equation}
As a consequence, the limsup in \eqref{eq-vol} is indeed a limit. 

Next we recall the notion of restricted volume (\cite{Tsu06, ELMNP09, BFJ09}) and the asymptotic intersection number that calculates the restricted volume. 
\begin{defn}[\cite{ELMNP09}]
For any irreducible ($d$-dimensional) subvariety $Z\subset \mcX$
\begin{itemize} 
\item The restricted volume of $\mcL$ along $Z$ is defined as
\begin{equation}\label{eq-resvol}
\vol_{\mcX|Z}(\mcL)=\limsup_{m\rightarrow+\infty}\frac{\dim_{\bC} {\rm Im}\left(H^0(\mcX, m\mcL)\rightarrow H^0(Z, m\mcL|_Z)\right)}{m^d/d!}.
\end{equation}
\item For any $Z\not\subseteq \bfB(\mcL)$ (the stable base locus of $\mcL$), the asymptotic intersection number of $\mcL$ and $Z$ is defined as:
\begin{equation}\label{eq-asymint}
\|\mcL^d \cdot Z\|:=\limsup_{m\rightarrow+\infty} \mcL_m^d\cdot \tilde{Z}_m,
\end{equation}
where $\tilde{Z}_m$ is the strict transform of $Z$ under the normalized blowup $\mu_m: \mcX_m\rightarrow \mcX$ of base ideal of $|m\mcL|$. 

\end{itemize}
\end{defn}
\begin{rem}
It is shown in \cite{ELMNP09} that the limsup in the formula \eqref{eq-resvol} and \eqref{eq-asymint} are actually  limits.
\end{rem}

Boucksom-Favre-Jonsson \cite{BFJ09} proved that the restricted volume is equal to a positive intersection product. 
\begin{defn}[{\cite[Definition 2.5]{BFJ09}}]\label{def-movint}
Let $\mcL$ be a big $\bQ$-line bundle. For any effective divisor $D$, define:
\begin{equation}\label{eq-movint}
\la \mcL^n\ra \cdot D=\sup_{\mu, E}\; (\mu^*L-E)^n\cdot \mu^*D,
\end{equation} 
where supremum is taken over all birational morphism $\mu: \tilde{\mcX}\rightarrow \mcX$ and an effective divisor $E$ such that $\mu^*L-E$ is nef. If $D=\sum_i b_i D_i$ with $b_i\in \bR$ with $D_i$ effective, then we extend the definition \eqref{eq-movint} linearly:
\begin{equation*}
\la \mcL^n\ra \cdot D=\sum_i b_i \la \mcL^n\ra\cdot D_i.
\end{equation*}

\end{defn}
 \begin{rem}
In \cite{BFJ09}, Boucksom-Favre-Jonsson defined positive intersection product $\la \xi^p\ra$ for any big class $\xi\in N^1(\mcX)_\bR$ and $1\le p\le n+1$, by developing an intersection theory on the Riemann-Zariski space. For example, when $p=n+1$, $\la \xi^{n+1}\ra=\vol(\xi)$; when $p=1$, $\la \xi\ra$ is the collection of positive parts of divisorial Zariski decomposition of $\pi^*\xi$ for all smooth blowups $\pi: \mcX_\pi\rightarrow \mcX$. We refer to \cite{BFJ09} for details on these more general definitions.

Moreover an analytic definition of the positive intersection product 
 was defined even earlier in \cite[Theorem 3.5]{BDPP13} (called movable intersection product there).
For each semipositive class $\alpha\in H^{1,1}(\mcX, \bR)$, define:
\begin{equation}\label{eq-intalp}
\la \mcL^n\ra\cdot \alpha=\sup_{T, \mu}\;\{\beta^n\cdot \mu^*\alpha\}
\end{equation}
where  $T$ ranges over all K\"{a}hler currents in $c_1(\mcL)$ that have logarithmic poles and $\mu: \tilde{\mcX}\rightarrow \mcX$ ranges over the set of those log resolutions satisfying $\mu^*T=\{E\}+\beta$ (with $\{E\}$ an effective divisor and $\beta$ smooth and semipositive).
By Poincar\'{e} duality the class $\la\mcL^n\ra$ is uniquely defined as a semipositive class in $H^{n,n}(\mcX, \bR)$. 
\end{rem}
In the above definitions, we see that the left-hand-side of \eqref{eq-movint} depends only on the numerical class of $\mcL$ and $D$.

Recall that the augmented base locus of $\mcL$ is defined as (see \cite{ELMNP09}):
\begin{equation}
\bfB_+(\mcL)=\bigcap_{\mcL=A+E} \supp(E),
\end{equation}
where the intersection is over all decompositions of $\mcL=A+E$ into $\bQ$-divisors with $A$ ample and $E$ effective. It is know that the augmented base locus depends only on the numerical class of $\mcL$ (see \cite{ELMNP09} and reference therein). 
We will use the following important results:
\begin{thm}\label{thm-restvol}
If $\mcL\rightarrow \mcX$ is a big line bundle, and $Z\subset X$ is a prime divisor, then the following statements are true: 
\begin{enumerate}
\item (\cite[Theorem 2.13, Theorem C]{ELMNP09}) If $Z\not\subseteq \bfB_+(\mcL)$ then $\vol_{\mcX|Z}(\mcL)=\|\mcL^n\cdot Z\|$. If $Z\subseteq \bfB_+(\mcL)$, then $\vol_{\mcX|Z}(\mcL)=0$. 

\item (\cite[Theorem B]{BFJ09}) There is an identity $\vol_{\mcX|Z}(\mcL)=\la \mcL^{n}\ra \cdot Z$. As a consequence, $\vol_{\mcX|Z}(\mcL)$ depends only on the numerical class of $\mcL$ and $Z$.
\end{enumerate}

\end{thm}

As a consequence of these results, we know that in the definition of positive intersection number in \eqref{eq-movint}, it suffices to take the supremum along the sequence 
$\mu_m: \mcX_m\rightarrow \mcX$ which is the normalized blowup of $\fb(|mL|)$ (or its resolution) with exceptional divisor $\tilde{E}_m$. In other words, if we set $\mcL_m=\mu_m^*\mcL-\frac{1}{m}\tilde{E}_m$, then for any divisor $D$, we have the identity:
\begin{equation}
\la \mcL^n\ra\cdot D=\lim_{m\rightarrow+\infty}\mcL_m^n\cdot \mu^*D.
\end{equation}

\section{Positive intersection formula}\label{sec-mov}


Let $\pi: (\mcX, \mcL)\rightarrow \bC$ be a {\it big} model of $(X, L)$. By resolution of singularities, we can assume that $(\mcX, \mcX^{\rm red}_0)$ is a dominating and SNC model of $(X, L)$.  
From now on, for simplicity of notation, we still denote by $(\mcX, \mcL)$ its natural compactification over $\bP^1$.

Because $\mcL$ is big over $\mcX(=\bar{\mcX})$ (by the definition of big model), $\bfB_+(\mcL)\neq \mcX$, there exists a fiber $\mcX_t=\pi^{-1}(\{t\})$ for some $t\in \bP^1\setminus \{0\}$ such that $\mcX_t\not\subseteq \bfB_+(\mcL)$. In particular, $\mcX_t\not\subseteq \bfB(\mcL)$. 
We then apply Theorem \ref{thm-restvol} to get 
\begin{equation}\label{eq-volXt}
\la \mcL^n\ra \cdot \mcX_t=\vol_{\mcX|X}(\mcL)=\|\mcL^n\cdot X\|=V.
\end{equation}

Because $\la \mcL^n\ra \cdot \mcX_t$ depends only on numerical classes of $\mcL$ and $\mcX_t$ (see definition \ref{def-movint}), we can use $\mcX_t\equiv \mcX_0=\sum_{i=1}^I b_i E_i$ to get:
\begin{equation}\label{eq-totalV}
V=\la \mcL^n\ra \cdot \mcX_t=\sum_{i=1}^I b_i  \left(\la \mcL^n\ra \cdot E_i\right).
\end{equation}
Now we can prove the formula \eqref{eq-MAmodel} for the non-Archimedean Monge-Amp\`{e}re measure of non-Archimedean metrics associated to model filtrations. This result refines and generalizes \cite[Lemma 8.5]{BFJ15}.
\begin{proof}[Proof of Theorem \ref{thm-main}.(i)]
We will use the notations in section \ref{sec-model}. 
Via the resolution of singularity, we can first replace $\mcX$ by any SNC model $\mcX'$ that dominates $\mcX$ via $\pi': \mcX'\rightarrow \mcX$ and replace $\mcL$ by $\mcL'=\pi'^*\mcL$. For simplicity of notations, we will still use the notation $(\mcX, \mcL)$ instead of $(\mcX', \mcL')$.

Because the sequence of continuous metrics $\phi_m:=\phi_{(\mcX_m, \mcL_m)}\in \mcH^\NA$ increases to the continuous metric $\phi_{\mcL}=\phi_\triv+P(f_\mcL)$, by Dini's theorem we know that $\phi_{m}$ converges to $\phi_\mcL$ uniformly. In particular, $\phi_m$ converges to $\phi_\mcL$ in the strong topology and $\MA^\NA(\phi_m)$ converges strongly, and hence also weakly, to $\MA^\NA(\phi_\mcL)$.

Set $\nu_{\mcX, m}=(r_{\mcX})_*(\MA^\NA(\phi_m))$ and $\nu_{\mcX}=(r_{\mcX})_*\MA^\NA(\phi_\mcL)$ where $r_{\mcX}: X\rightarrow \Delta_{\mcX}$ is the natural retraction to the dual complex of $\mcX$ (see \cite{BoJ18a}). Then they are supported on $\Delta_\mcX$ and it is easy to see that $\nu_{\mcX,m}$ converges to $\nu_{\mcX}$ weakly. 
By Portmanteau's theorem for weak convergence of measures (see \cite[Theorem 2.1]{Bil99}), we have:
\begin{equation}\label{eq-Port1}
\limsup_{m\rightarrow+\infty} \nu_{\mcX, m}(\{x_i\})\le \nu_{\mcX}(\{x_i\}).
\end{equation}
On the other hand, we clearly have 
\begin{eqnarray*}
\nu_{\mcX,m}(\{x_i\})&=&(r_{\mcX})_*\MA^\NA(\phi_m)(\{x_i\})\\
&=&\MA^\NA(\phi_m)((r_{\mcX})^{-1}\{x_i\})\ge \MA^\NA(\phi_m)(\{x_i\}).
\end{eqnarray*}
So we combine the above two inequalities to get:
\begin{equation}\label{eq-Port}
\limsup_{m\rightarrow+\infty} \MA^\NA(\phi_m)(\{x_i\})\le \nu_{\mcX}(\{x_i\})=:V_i.
\end{equation}
 We consider two cases:
\begin{enumerate}
\item If $E_i\not\subseteq \bfB_+(\mcL)$, then $E_i$ is not contained in $\bfB(\mcL)$. 
By the formula of non-Archimedean Monge-Amp\`{e}re measures of test configurations (see \cite[section 3.4]{BoJ18a}) 
we get that:
\begin{equation}\label{eq-auxineq}
b_i \mcL_m^n\cdot \tilde{E}_i=
\MA^\NA(\phi_{m})(\{x_i\}),
\end{equation}
where $\tilde{E}_i$ is the strict transform of $E_i$ under $\mu_m$.
So by \eqref{eq-auxineq} and \eqref{eq-asymint} we get
\begin{equation}
\limsup_{m\rightarrow+\infty} \MA^\NA(\phi_m)(\{x_i\})=\limsup_{m\rightarrow+\infty} b_i \mcL_m^n\cdot \tilde{E}_i=b_i \|\mcL^n\cdot E_i\|.
\end{equation}
So by Theorem \ref{thm-restvol} and the inequality \eqref{eq-Port} we have 
\begin{equation}
b_i \la \mcL^n\ra\cdot E_i=b_i\cdot \vol_{\mcX|E_i}(\mcL)=
b_i \|\mcL^n\cdot E_i\|\le V_i.
\end{equation}

\item If $E_i\subseteq \bfB_+(\mcL)$, then $b_i \la \mcL^n\ra \cdot E_i=0\le V_i$.
\end{enumerate}
Combining these with \eqref{eq-totalV}, we have:
\begin{eqnarray*}
V&=&\sum_i b_i \la \mcL^n\ra \cdot E_i\le \sum_i V_i=\sum_i \nu_{\mcX}(\{x_i\}) \le V.
\end{eqnarray*}
So the inequalities in the above chain are actually equalities. So $b_i \la \mcL^n\ra \cdot E_i=V_i=\nu_{\mcX}(\{x_i\})$ for $i=1, \dots, I$ and $\nu_{\mcX}=(r_{\mcX})_*\MA^\NA(\phi_\mcL)$ is supported on the finite set $\{x_i; i=1,\dots, I\}$. In other words, we have
\begin{equation}
(r_{\mcX})_*\MA^\NA(\phi_\mcL)=\sum_{i=1}^N b_i \la\left(\mcL^n\ra\cdot E_i \right) \delta_{x_i}.
\end{equation}
But we have said that $(\mcX, \mcL)$ can be replaced by any SNC model that dominates $\mcX$. Moreover, the pairs $\{(x_i, V_i); V_i\neq 0\}$ do not depend on the choice of such SNC models. By using the homemorphism $X^\NA=\varprojlim \Delta_\mcX$, it is then easy to conclude that the Radon measure $\MA^\NA(\phi_\mcL)$ is indeed only supported on the finite set $\{x_i; i=1,\dots, I\}$ and the identity \eqref{eq-MAmodel} holds true.
\end{proof}
\begin{rem}
Although our work on K-stability is the through the study of non-Archimedean geometry in the trivially valued case (which is base-changed to the discretely valued case, following \cite{BoJ18a, BoJ18b}), 
 the proof of formula for non-Archimedean Monge-Amp\`{e}re measure also holds true for more general non-trivially valued case.
\end{rem}

We recall that the formula for non-Archimedean functionals following the works in \cite{BHJ17, BoJ18a, BoJ18b} (see also \cite{Li20}). For any continuous continuous psh metric $\phi$ on $L^\NA$, the non-Archimedean Mabuchi functional is given by:
\begin{equation}\label{eq-MNA}
\bfM^\NA(\phi)=\bfH^\NA(\phi)+(\bfE^{K_X})^\NA(\phi)+\ud{S}\; \bfE^\NA(\phi)
\end{equation}
where the terms on the right-hand-side are given by the following non-Archimedean integrals:
\begin{eqnarray*}
\bfH^\NA(\phi)&=&\int_{X^\NA}A_X(x)\MA^\NA(\phi), \\
(\bfE^{K_X})^\NA(\phi)&=& \sum_{i=0}^{n-1} \int_{X^\NA} (\phi-\phi_\triv) \ddc\psi\wedge \MA^\NA(\phi_\triv^{[i]}, \phi^{[n-1-i]})\\
\bfE^\NA(\phi)&=&\frac{1}{n+1}\sum_{i=0}^{n+1}\int_{X^\NA} (\phi-\phi_\triv) \MA^\NA(\phi_\triv^{[i]}, \phi^{[n-i]}),
\end{eqnarray*}
where in the second identity $\psi$ is a Hermitian metric on $K_X^{\NA}$. We also recall the $\bfJ^\NA$-functional:
\begin{equation}\label{eq-JNA}
\bfJ^\NA(\phi)=L^n\cdot \sup(\phi-\phi_\triv)-\bfE^\NA(\phi).
\end{equation} 
\begin{prop}
With the above notation, we have:
\begin{equation}\label{eq-Hmodel}
\bfH^\NA(\phi_\mcL)=\la \mcL^n\ra\cdot K^{\log}_{\mcX/X_{\bP^1}}.
\end{equation}
\end{prop}
\begin{proof}
Note that we have the identity:
\begin{eqnarray*}
K^{\log}_{\mcX/X_{\bP^1}}&=&K_{\mcX}+\mcX^{\rm red}_0-(K_{X_{\bP^1}}+\mcX_0)=\sum_i (A_{X_{\bP^1}}(E_i)-b_i)E_i\\
&=&\sum_i b_i (A_{X_{\bP^1}}(b_i^{-1}\ord_{E_i})-1)E_i=\sum_i b_i A_X(x_i)E_i.
\end{eqnarray*}
So we can use \eqref{eq-MAmodel} to get the identity:
\begin{eqnarray*}
\bfH^\NA(\phi_\mcL)&=&\int_{X^\NA} A_X(x) \MA^\NA(\phi_\mcL)=\sum_i A_X(x_i) b_i \la \mcL^n\ra \cdot E_i\\
&=&\la \mcL^n\ra \cdot K^{\log}_{\mcX/X_{\bP^1}}.
\end{eqnarray*}
\end{proof}

\begin{prop}\label{prop-REphi}
With the above notation, we have the following identities:
\begin{eqnarray}
\bfR^\NA(\phi_\mcL)&=&\la \mcL^n\ra \cdot \rho^*K_X, \label{eq-RNAPphi}\\
\bfE^\NA(\phi_\mcL)&=&\frac{1}{n+1}\la \mcL^{n+1}\ra=\frac{1}{n+1}\la \mcL^n\ra \cdot \mcL. \label{eq-ENAPphi}
\end{eqnarray}
\end{prop}

\begin{proof}
Because $\phi_m$ converges to $\phi_\mcL$ strongly, by \cite{BoJ18a} 
we have:
\begin{equation}\label{eq-RNAm}
\bfR^\NA(\phi_\mcL)=\lim_{m\rightarrow+\infty} \bfR^\NA(\phi_m)=\lim_{m\rightarrow+\infty} \mcL_m^n\cdot \mu_m^*\rho^* K_X.
\end{equation}
Write $\rho^*K_X=A_1-A_2$ with $A_1, A_2$ very ample. Moreover we can choose $A_i, i=1,2$  to be sufficient general such that $A_i, i=1, 2$ do not contain the centers of Rees valuations of $\fb_m$ for all $m$. Then the strict transforms of $A_i, i=1, 2$ under $\mu_m: \mcX_m\rightarrow \mcX$ are the same as the total transform of $A_i, i=1,2$. By using Theorem \ref{thm-restvol} we see that the right-hand-side of \eqref{eq-RNAm} is equal to  
\begin{equation}
\|\mcL^n\cdot A_1\|-\|\mcL^n\cdot A_2\|=\la \mcL^n\ra \cdot (A_1-A_2)=\la \mcL^n \ra\cdot \rho^*K_X.
\end{equation}
For the first equality in \eqref{eq-ENAPphi}, we can again use $\phi_m=\phi_{(\mcX_m, \mcL_m)}$ (for which \eqref{eq-ENAPphi} is known to be true) to approximate and directly apply the Fujita approximation result in \cite[Theorem 11.4.11]{Laz04}. The last equality in \eqref{eq-ENAPphi} follows from the orthogonality property proved in
\cite[Corollary 4.5]{BDPP13} or \cite[Corollary 3.6]{BFJ09}.

\end{proof}

We can complete the proof the formula for the non-Archimedean Mabuchi functional.
\begin{proof}[Proof of Theorem \ref{thm-main}.(ii)] 
The formula \eqref{eq-Mmodel} follows immediately from the decomposition $\bfM^\NA=\bfH^\NA+\bfR^\NA+\ud{S}\bfE^\NA$  and the formula for each part in \eqref{eq-Hmodel}, \eqref{eq-RNAPphi} and \eqref{eq-ENAPphi}.

\end{proof}

As an application of the positive intersection formula, we get:
\begin{prop}\label{prop-reduced}
To check the ($G$-) unform K-stability for models (see Definition \ref{def-uniK} and \cite{Li20}), it suffices to consider models with reduced central fibres.
\end{prop}
\begin{proof}
Let $(\mcX, \mcL)$ be any big model. We can take a base change $(\mcX^{(d)}, \mcL^{(d)})=(\mcX, \mcL)\times_{\bC, t\mapsto t^d}\bC$ such that its normalization $\tilde{\mcX}$ has reduced central fibers. Let $f: \tilde{\mcX}\rightarrow \mcX$ be the natural finite morphism and set $\tilde{\mcL}=f^*\mcL$. Then we have the identity 
$$K^{\log}_{(\tilde{\mcX}, \tilde{\mcX}_0)}:=K_{\tilde{\mcX}}+\tilde{\mcX}_0=f^*(K_{\mcX}+\mcX_0^{\rm red})=f^*K^{\log}_{(\mcX, \mcX_0)}.$$ 

It is known that volumes of big line bundles are multiplicative under generic finite morphisms (see \cite[Lemma 4.3]{Hol10}). So we get the identity 
\begin{equation}\label{eq-fmult1}
\la (\tilde{\mcL}+\epsilon K^{\log}_{(\tilde{\mcX}, \tilde{\mcX}_0)})^{n+1}\ra=d \cdot \la (\mcL+\epsilon K^{\log}_{(\mcX, \mcX_0)})^{n+1}\ra.
\end{equation}
Taking derivative with respect to $\epsilon$ at $\epsilon=0$, we also get:
\begin{equation}\label{eq-fmult2}
\la \tilde{\mcL}^{n}\ra\cdot K^{\log}_{(\tilde{\mcX}, \tilde{\mcX}_0)}=d\cdot \la \mcL^{n}\ra \cdot K^{\log}_{(\mcX, \mcX_0)}.
\end{equation}
Moreover \eqref{eq-fmult1} for $\epsilon=0$ gives $\bfE^\NA(\phi_{\tilde{\mcL}})=d\cdot \bfE^\NA(\phi_\mcL)$. 
On the other hand, it is known we have the formula (see \cite{BoJ18a}) 
\begin{equation}
(\phi_{\tilde{\mcL}}-\phi_\triv)(x)=d\cdot (\phi_\mcL-\phi_\triv)(d^{-1}x), \quad \text{ for all } x\in X^\NA.
\end{equation}
So we get the identity $\bfJ^\NA(\phi_{\tilde{\mcL}})=d\cdot \bfJ^\NA(\phi_{\mcL})$ by \eqref{eq-JNA}.
Combining these identities with the positive intersection formula \eqref{eq-Mmodel}, the statement now follows easily.

\end{proof}

\section{First Riemann-Roch coefficients of big line bundles and Fujita approximations}\label{sec-conj}

In view of the above intersection formula, it seems natural to consider the following invariant for big line bundles.

\begin{defn}
Let $\mcL$ be a big line bundle over a projective manifold $\mcX$ of dimension $n+1$. The first Riemann-Roch coefficient (1st-RR coefficient) of $\mcL$ is defined to be:
\begin{equation}
\mathfrak{r}_1(\mcX, \mcL)=\la \mcL^n\ra \cdot K_{\mcX}.
\end{equation}
If the base manifold $\mcX$ is clear, we just write $\mathfrak{r}_1(\mcX, \mcL)$ as $\mathfrak{r}_1(\mcL)$.
\end{defn}
The zero-th Riemann-Roch coefficient is of course the volume of $\mcL$:
\begin{equation}
\mathfrak{r}_0(\mcX, \mcL):=\vol_{\mcX}(\mcL)=\la \mcL^{n+1}\ra.
\end{equation}
One would hope that $\mathfrak{r}_1(\mcX, \mcL)$ is the second order coefficients in the expansion of $h^0(\mcX, m\mcL)$. This is true if $\mcL$ is big and nef by Fujita's vanishing theorem. But due to the example in \cite{CS93}, this does not seem to be true for general big line bundles.
\begin{lem}\label{lem-pullinv}
If $\mu: \mcY\rightarrow \mcX$ is a birational morphism between smooth projective manifold, which is a composition of blowups along smooth subvarieties. Then we have:
\begin{equation}\label{eq-binv}
\mathfrak{r}_1(\mcL)=\mathfrak{r}_1(\mu^*\mcL).
\end{equation}
\end{lem}
\begin{proof}
Write $K_{\mcX}$ as the difference of very ample divisors $A_1-A_2$ and arguing as in the proof of Proposition \ref{prop-REphi}, we see that:
\begin{equation}
\la \mcL^n\ra \cdot K_{\mcX}=\la \mu^*\mcL^n\ra \cdot \mu^*K_{\mcX}.
\end{equation}
Let $E_i$ be the exceptional divisor of $\mu$. We just need to show that $\la \mu^*\mcL^n\ra\cdot E_i=\vol_{\mcY|E_i}(\mu^*\mcL^n)=0$. This can be seen by the inclusion:
\begin{equation}
{\rm Im}\left(H^0(\mcY, m\mu^*\mcL)\rightarrow H^0(E_i, m\mu^*\mcL|_{E_i})\right)\subseteq H^0(E_i, m \mu^*\mcL|_{E_i})=H^0(\mu_*(E_i), m\mcL|_{\mu_*(E_i)})
\end{equation}
and using the fact that the right-hand-side is equal to $o(m^{n})$ because $\dim(\mu_*(E_i))<n$.  
\end{proof}
If we consider $\mcL$ as a Cartier $b$-divisor in the sense of Shokurov, then because of identity \eqref{eq-binv}, $\mathfrak{r}_1(\mcL)$ is an invariant of the Cartier $b$-divisor $\mcL$.

The following lemma follows immediately from the results in \cite[section 3.1]{Mat13}.
\begin{lem}\label{lem-Zar}
Let $\mcL=\mcP+\mcN$ be the divisorial Zariski decomposition of $\mcL$. Then we have:
\begin{equation}
\mathfrak{r}_1(\mcL)=\mathfrak{r}_1(\mcP).
\end{equation}
Moreover if $\mcL$ admits a Zariski decomposition, i.e. if $\mcP$ is nef, then we have:
\begin{equation}\label{eq-e1nef}
\mathfrak{r}_1(\mcL)=\mcP^{n}\cdot K_{\mcX}.
\end{equation}
\end{lem}
We propose the following main conjecture. 
\begin{conj}\label{conj-Fujmodel}
Let $(\mcX, \mcL)$ be a {\it big} model of $(X, L)$. Then there exists a sequence of blowups $\mu_m: \mcX_m\rightarrow \mcX$ along $\bC^*$-equivariant ideal sheaves cosupported on $\mcX_0$ and decompositions into $\bQ$-divisors 
$\mu_m^*\mcL=\mcL_m+E_m$ with $\mcL_m$ semiample and $E_m$ effective supported on the exceptional divisor of $\mu_m$ such that:
\begin{equation}\label{eq-RRconv}
\lim_{m\rightarrow+\infty} \vol_{\mcX_m}(\bar{\mcL}_m)=\vol_{\mcX}(\bar{\mcL}) \quad \text{ and } \quad
\lim_{m\rightarrow+\infty} \mathfrak{r}_1(\bar{\mcL}_m)=\mathfrak{r}_1(\bar{\mcL}). 
\end{equation}
\end{conj}
Because of the positive intersection formula in \eqref{eq-Mmodel} and the reduction in \cite{Li20}, this indeed implies Boucksom-Jonsson's regularization conjecture. Moreover by the following lemma and the work in \cite{Li20}, it would  complete the solution of Yau-Tian-Donaldson conjecture for cscK metrics.
\begin{lem}\label{lem-conjBJ}
For any big model $(\mcX, \mcL)$, 
Conjecture \ref{conj-Fujmodel} implies that there exists $\phi_m\in \mcH^\NA$ such that $\phi_m$ converges to $\phi_{(\mcX, \mcL)}$ in the strong topology and $\bfM^\NA(\phi_m)\rightarrow \bfM^\NA(\phi_{(\mcX, \mcL)})$. 
\end{lem}
\begin{proof}
By the same base change construction as in the proof of Proposition \ref{prop-reduced}, we can assume that $\mcX$ has a reduced central fibre.

For simplicity of notations, we denote by $\phi=\phi_{(\mcX, \mcL)}$ (resp. $\phi_m$) the non-Archimedean metrics associated to $\mcL$ (resp. $\mcL_m$). Then because $E_m$ is effective, we have $\phi\ge \phi_m$. We claim that $\phi_m\rightarrow \phi$ strongly. Indeed, by \cite[Proposition 6.26]{BoJ18b}, it suffices to show the following non-negative quantity converges to 0 as $m\rightarrow+\infty$:
\begin{equation}
\bfJ^\NA_\phi(\phi_m)=\int_{X^\NA}(\phi_m-\phi)\MA^\NA(\phi_m)-\bfE^\NA(\phi_m)+\bfE^\NA(\phi).
\end{equation}
This follows immediately from $\phi_m\le \phi$ and \eqref{eq-ENAPphi}:
\begin{eqnarray*}
0\le \bfJ^\NA_\phi(\phi_m)\le -\bfE^\NA(\phi_m)+\bfE^\NA(\phi)=\frac{\vol_{\mcX}(\mcL)}{n+1}-\frac{\vol_{\mcX_m}(\mcL_m)}{n+1}.
\end{eqnarray*}


By the positive intersection formula \eqref{eq-Mmodel} the second identity in \eqref{eq-RRconv} implies $\bfM^\NA(\phi_m)\rightarrow \bfM^\NA(\phi)$.
\end{proof}
We hope the conjecture \ref{conj-Fujmodel} can be studied by using the geometric tools introduced in the study of Fujita's approximation theorem.
We recall the following definition 
\begin{defn}[{see \cite[Definition 11.4.3]{Laz04}}]
Let $\mcL$ be a big line bundle. A Fujita approximation of $\mcL$ consists of a projective birational morphism $\mu: \mcX'\rightarrow \mcX$ with $\mcX'$ irreducible together with a decomposition $\mu^*\mcL=A+E$ in $N^1(\mcX)_\bQ$ such that $A$ is big and semiample and $E$ is effective. 
\end{defn}
In the more general context of big line bundles, we conjecture the following result:
\begin{conj}\label{conj-Fuj}
Let $\mcL$ be a big line bundle over a smooth projective manifold $\mcX$. Then there exists a sequence of Fujita approximations (\cite[Definition 11.4.3]{Laz04}), i.e. birational morphisms  $\mu_m: \mcX_m\rightarrow \mcX$ and decompositions into $\bQ$-divisors $\mu_m^*\mcL=\mcL_m+E_m$ with $\mcL_m$ ample and $E_m$ effective,  such that:
\begin{equation}\label{eq-RRconv2}
\lim_{m\rightarrow+\infty} \vol_{\mcX_m}(\mcL_m)=\vol_{\mcX}(\mcL)\quad \text{ and }\quad
\lim_{m\rightarrow+\infty} \mathfrak{r}_1(\mcL_m)=\mathfrak{r}_1(\mcL).
\end{equation}
\end{conj}
\begin{rem}\label{rem-Boucksom}
Sebastien Boucksom pointed out to me that this conjecture could be formulated using the language of b-divisors. Such a formulation has some consequences and (hopefully) might be useful for studying this problem.
\end{rem}

Let's recall an orthogonality estimate by Boucksom-Demailly-P\v{a}un-Peternell (see also \cite[Theorem 11.4.21]{Laz04}):
\begin{thm}[{\cite[Theorem 4.1]{BDPP13}}]\label{thm-BDPP}
Fix any ample line bundle $H$ on $\mcX$. There exists a constant $C=C(\mcX, H)>0$ such that
any Fujita decomposition $(\mu: \mcX'\rightarrow \mcX, \mu^*\mcL=A+E)$ satisfies the estimate:
\begin{equation}
(A^n\cdot E)^2\le C\cdot (\vol_{\mcX}(\mcL)-\vol_{\mcX'}(A)).
\end{equation}
\end{thm}
We observe an immediate consequence of this estimate.
\begin{lem}\label{lem-Abd}
Let $\mu_m: \mcX_m\rightarrow \mcX$ be a sequence of birational morphisms such $\mu_m^*\mcL=\mcL_m+E_m$ where $\mcL_m$ is ample and $E_m$ is effective. Assume that the following conditions are satisfied:
\begin{enumerate}
\item $\lim_{m\rightarrow+\infty} \vol_{\mcX_m}(\mcL_m)=\vol_\mcX(\mcL)$.
\item $\lim_{m\rightarrow+\infty}\mcL_m^n\cdot \mu_m^*K_{\mcX}=\la \mcL^n\ra\cdot K_\mcX$.
\end{enumerate}
Then $\lim_{m\rightarrow+\infty}\fr_1(\mcL_m)=\fr_1(\mcL)$ if and only if
\begin{equation}\label{eq-cond3}
\lim_{m\rightarrow+\infty}\mcL_m^n\cdot K_{\mcX_m/\mcX}=0.
\end{equation}
In particular, if there exists a constant $C>0$ independent of $m$ such that for any irreducible component $F$ of $E_m$ we have $\ord_F(K_{\mcX_m/\mcX})\le C\cdot \ord_F(E_m)$, then we have the convergence:
$\lim_{m\rightarrow+\infty} \mathfrak{r}_1(\mcL_m)=\mathfrak{r}_1(\mcL)$.
\end{lem}
\begin{rem}
The above lemma suggests that the techniques from birational algebraic geometry might be useful for achieving \eqref{eq-cond3}. 
Indeed, our hope is that the MMP techniques (based on the work of Birkar-Casini-Hacon-McKernan) could be used to extract suitable exceptional divisors satisfying the conditions in the above lemma. Note that such type of techniques has prove to be very powerful in the study of K-stability for Fano varieties (see for example \cite{BLX19}). 

\end{rem}

By the works in \cite{DEL00, Laz04} and \cite{ELMNP09, BFJ09}, the sequence $\{\mcL_m\}$ that satisfy the first two conditions can be obtained by blowing up base ideals. Moreover one can also get $\mcL_m$ by blowing up appropriate asymptotic multiplier ideals, which satisfy the important Nadel-vanishing and global generation properties. We review the construction in \cite[11.4.B, Proof of Theorem 11.4]{Laz04} for the reader's convenience. Fix a very ample bundle $H$ on $\mcX$ such that $G:=K_\mcX+(n+2)H$ is very ample. For $m\ge 0$, set $M_m=m\mcL-G$. Given $\epsilon>0$ there exists $m\gg 1$ such that $\vol(M_m)\ge m^{n+1} (\vol(\mcL)-\epsilon)$. Set $\cJ=\cJ(\mcX, \|M_m\|)$, let $\mu_m: \mcX_m\rightarrow \mcX$ be a common resolution of $\cJ$ such that $\mu_m^*\cJ=\mcO(-\tilde{E}_m)$. Then $\mcL_m:=\mcL-\frac{1}{m}\tilde{E}_m$ is semiample and $\mcL_m^{n+1}\ge \vol(\mcL)-\epsilon$ (see \cite[11.4.B]{Laz04} for more details). By letting $\epsilon\rightarrow 0$, we see that the first condition in Lemma \ref{lem-Abd} is thus satisfied. 

Now we claim that in this construction, the second condition in Lemma \ref{lem-Abd} can also be satisfied. This fact will be used in the calculations of appendix \ref{sec-app}. To see this, we use some similar argument as in \cite[Proof of Theorem 2.13]{ELMNP09}. Choose $m_0\gg 1$ such that $m_0 \mcL-G=N'$ is effective. Fix a very ample divisor $H$ such that $N:=N'+H$ is ample. Then we have the inclusion
\begin{eqnarray*}
\fb(|(m-m_0)\mcL|)\mcO_\mcX(-N)&\subseteq& \fb(|(m-m_0)\mcL)\mcO(-N')\\
&\subseteq& \fb(|m\mcL-G|)=\fb(|M_m|)\subseteq \cJ(\|M_m\|).
\end{eqnarray*}
We can also assume that $\mu_m$ is both resolutions of $\fb(|M_m|)$ and $\fb(|(m-m_0)\mcL|)$ satisfying the identities $\mu_m^*\fb(|m\mcL-G|)=\mcO_{\mcX_m}(-\tilde{F}_m)$ and $\mu_m^*\fb(|(m-m_0)\mcL|)=\mcO_{\mcX_m}(-\tilde{Q}_m)$. Set $\mcL'_m=\mu_m^*(\mcL-\frac{G}{m})-\frac{1}{m}\tilde{F}_m$ and $\mcL''_m=\mu_m^*\mcL-\frac{1}{m-m_0}\tilde{Q}_m$. 

Fix any effective divisor $D$ on $\mcX$. Let $\tilde{D}_m$ be the strict transform of $D$ under $\mu_m$.  Then the above inclusion implies:
$$
\vol((\mcL_m+\frac{N}{m})|_{\tilde{D}_m})\ge \vol((\mcL'_m+\frac{N}{m})|_{\tilde{D}_m})\ge \vol(\mcL''_m|_{\tilde{D}_m}).
$$
Because $\mcL_m$ and $\mcL''_m$ are both semiample, this implies
\begin{equation}\label{eq-estm}
(\mcL_m+\frac{N}{m})^n\cdot \tilde{D}_m\ge \mcL''^n_m \cdot \tilde{D}_m.
\end{equation}
Now we can fix a very ample line bundle $H$ such that $\mu_m^*(mH)-\mcL_m=m \mu^*(H-\mcL)+\tilde{E}_m$ is effective.
Then for any $1\le i\le n$, we have:
$$
\limsup_{m\rightarrow+\infty} \frac{1}{m^{n}}(m\mcL_m^{n-i})\cdot N^i \cdot \tilde{D}_m\le \limsup_{m\rightarrow+\infty} \frac{1}{m^n} (mH)^{n-i}\cdot N^i \cdot \tilde{D}_m=0.
$$
By expanding the left-hand-side of \eqref{eq-estm}, this implies
\begin{equation*}
\la \mcL^n\ra\cdot D\ge \limsup_{m\rightarrow+\infty} \mcL_m^n \cdot \tilde{D}_m\ge \limsup_{m\rightarrow+\infty} \mcL''^n_m \cdot \tilde{D}_m=\|\mcL^n\cdot D\|
\end{equation*}
which, by using Theorem \ref{thm-restvol}, implies the equality 
$$\lim_{m\rightarrow+\infty} \mcL_m^n\cdot \mu_m^*D=\|\mcL^n\cdot D\|=\la \mcL^n\ra\cdot D.$$
Writing $-K_{\mcX}=D_1-D_2$ with $D_1, D_2$ effective, we then see that the second condition of Lemma \ref{lem-Abd} is satisfied too. 


Finally we point out that Conjecture \ref{conj-Fuj} holds true any for any big line bundle that admits a birational Zariski decomposition (in the sense of Cutkosky-Kawamata-Moriwaki). Unfortunately not all big line bundles admit such birational Zariski decomposition by the counterexamples of Nakayama \cite{Nak04}. On the other hand, we verify in the appendix that conjecture \ref{conj-Fuj} indeed holds for some of Nakayama's examples. 
Indeed, we will show that in these examples the bound of discrepancies in the above lemma is indeed satisfied.
So it seems to be very interesting to know whether \eqref{eq-cond3} can be achieved in general. 
\begin{defn}
We say a big line bundle $\mcL$ admits a birational Zariski decomposition if there is a modification $\mu: \tilde{\mcX}\rightarrow \mcX$, a nef $\bR$-divisor $\mcP$ and an $\bR$-effective divisor $\mcN$ on $\tilde{\mcX}$ with the following properties:
\begin{itemize}
\item $\mu^*\mcL=\mcP+\mcN$.
\item For any positive integer $m>0$, the map
\begin{equation}
H^0(\tilde{\mcX}, \mcO_{\tilde{\mcX}}(\lfloor mP\rfloor))\rightarrow H^0(\tilde{\mcX}, \mcO_{\tilde{\mcX}}(m\mcL))
\end{equation}
induced by the section $e_m$ is an isomorphism, where $e_m$ is the canonical section of $\lceil m \mcN\rceil$.
\end{itemize}
\end{defn}
\begin{lem}
If a big line bundle $\mcL$ admits a birational Zariski decomposition, then the conjecture \ref{conj-Fuj} for $\mcL$ is true.
\end{lem}
\begin{proof}
By Lemma \ref{lem-pullinv} and \eqref{eq-e1nef}, we have $\mathfrak{r}_1(\mcL)=\mathfrak{r}_1(\mu^*\mcL)=\mathfrak{r}_1(\mcP)=\mcP^n\cdot K_{\tilde{\mcX}}$. Choose any ample divisor $A$ on $\tilde{\mcX}$. Because $\mcP$ is big and nef, we know that 
for $k\gg 1$, $k\mcP-A=\Delta_k$ is effective. So we get:
\begin{equation}
(m+k)\mcP=m\mcP+A+E_k,
\end{equation}
which implies the decomposition over $\tilde{\mcX}$:
\begin{equation}
\mu^*\mcL=\mcP+\mcN=\frac{1}{m+k}(m\mcP+A)+\frac{1}{m+k}\Delta_k+\mcN.
\end{equation}
By perturbing the coefficients of $A$, we can assume that $m\mcP+A$ is a $\bQ$-divisor. 
Set $\mcL_m=\frac{1}{m+k}(m\mcP+A)$. Then it is easy to see that \eqref{eq-RRconv2} holds true. 

\end{proof}

\begin{rem}\label{rem-spherical}
If $(X, L)$ is a polarized spherical manifold, it is known that its models in the sense of Definition \ref{def-model} is a Mori dream space (see the appendix A by Y. Odaka to \cite{Del20}). 
Since Zariski decomposition of big lines bundles always exist on Mori dream spaces, the above lemma in the $\bC^*$-equivariant setting gives an explanation why the Yau-Tian-Donaldson conjecture holds for polarized spherical manifolds. See \cite[Appendix A]{Del20} for a slightly different proof of this fact (again based on Theorem \ref{thm-YTD}).
\end{rem}



\appendix

\section{Conjecture \ref{conj-Fuj} for Nakayama's examples without birational Zariski decomposition}\label{sec-app}

In this appendix, we will use Lemma \ref{lem-Abd} to show that conjecture \ref{conj-Fuj} is indeed true for some examples of big line bundles that do not have a birational Zariski decomposition. Such examples were first discovered by Nakayama \cite{Nak04}. Here we will do a case study based on the construction of Fujita approximation in \cite[Theorem 11.4.4]{Laz04} and the calculation of asymptotic multiplier ideals for Nakayama's examples in the work of Koike \cite{Koi15}.

We first write down some notations. Set $S=E\times E$ for an elliptic curve $E$ without complex multiplication. Then the pseudoeffective cone ${\rm PE}(S)$ coincides with the nef cone ${\rm Nef}(S)$. Fix a point $p\in E$ and consider in $N^1(S)_\bR$ three classes:
\begin{equation*}
f_1=[\{P\}\times E]=:[F_1], \quad f_2=[E\times \{P\}]=:[F_2], \quad \delta=[\Delta]
\end{equation*}
where $\Delta\subset E\times E$ is the diagonal. Then $N^1(S)_\bR$ is spanned by $\{f_1, f_2, \delta\}$ and the description of the nef cone is known (see \cite[Lemm 1.5.4]{Laz04a}): $\alpha=x\cdot f_1+y\cdot f_2+z\cdot \delta\in N^1(S)_\bR$ is nef if and only if 
\begin{equation}
xy+xz+yz\ge 0, \quad x+y+z\ge 0.
\end{equation}

By standard linear algebra, we can use the following linear transformation to diagonalize the above relation:
\begin{eqnarray*}
&&l_1=\frac{1}{6}(f_1+f_2-2\delta), \quad l_2=\frac{1}{6}(-\sqrt{3}f_1+\sqrt{3}f_2), \quad \frac{1}{6}(f_1+f_2+\delta)\\
&&a=x+y-2z, \quad b=-\sqrt{3}x+\sqrt{3}y, \quad c=2(x+y+z).
\end{eqnarray*}
such that $\alpha=a l_1+b l_2+c l_3\in N^1(S)_\bR$ is nef if and only if 
\begin{equation}\label{eq-nefcone}
c^2\ge a^2+b^2, \quad c\ge 0.
\end{equation}
Let $L_i, i=0,1,2$ be three line bundles over $S$. Set 
\begin{equation}
X=\bP(\cO_S\oplus (L_1-L_0)\oplus (L_2-L_0)) \cong \bP(L_0\oplus L_1\oplus L_2).
\end{equation}
Denote by $H=\cO_X(1)$ the tautological line bundle. 

We use the description of $X$ as a toric bundle over $S$ as in \cite{Nak04}.
Let $\Sigma$ denote the standard fan of $\bP^2$, i.e. the fan generated by three cones:
\begin{equation*}
\sigma_1={\rm Cone}\{e_1, e_2\}, \quad \sigma_2={\rm Cone}\{e_2, -(e_1+e_2)\}, \quad \sigma_3={\rm Cone}\{-(e_1+e_2), e_1\}.
\end{equation*}
Let $h: \bR^2\rightarrow \bR$ be the piecewise linear function on $\Sigma$ satisfying $h(e_1)=h(e_2)=0$ and $h(-(e_1+e_2))=-1$.
Then $X$ is the toric bundle associated to $\Sigma$ and $h$ determines the line bundle $H$. 
Set $\mathbb{L}=\pi^*L_0+D_h$. By a result of Cutkosky (see \cite[Lemma 6.1]{Koi15}), $\bbL$ is a big line bundle if and only if there exists $(k_0,k_1,k_2)\in \bN^3$ such that $L_0^{k_0}\otimes L_1^{k_1}\otimes L_2^{k_3}$ is an ample line bundle over $S$. 
Moreover it is well-known that the canonical line bundle of the projective bundle $X$ is given by:
\begin{equation}
K_X=\pi^*(K_S+(L_1-L_0)+(L_2-L_0))-3 H=L_1+L_2-2L_0-3H.
\end{equation}
The last identity uses the triviality of $K_S$. 

We will consider the example in \cite[Example 6.5]{Koi15}. Set $L_0=4 F_1+4 F_2+\Delta$, $L_1=\mcO_V$, $L_2=\mcO_V(-F_1+9 F_2+\Delta)$. Then
\begin{equation}\label{eq-c1Li}
c_1(L_0)=6(l_1+3l_3), \quad c_1(L_1)=0, \quad c_1(L_2)=6 l_1+10\sqrt{3} l_2+18 l_3.
\end{equation}
Because  $c_1(L_0)$ is in the interior of the nef cone, $L_0$ is ample. Note that $H$ is relatively ample. 
So it is easy to see that there exist $a, b\in \bZ_{>0}$ such that $\frac{a L_0+b H-K_X}{n+1}$ is a very ample line bundle. Set $G=a L_0+b H$ and 
\begin{eqnarray*}
M_p&=&p\, \bbL-G=p (L_0+H)-(a L_0+bH)=(p-a)L_0+(p-b)H\\
&=&(p-b)\left(\frac{p-a}{p-b}L_0+H\right)=:(p-b)(Q_p+H). 
\end{eqnarray*}
Set $\cJ_p:=\cJ(X, \|M_p\|)$. 
Let $\mu_p: Y_p\rightarrow X$ be the normalized blowup of $\cJ_p$ with $E_p:=\mu_p^*\cJ_p=\mcO_{Y_p}(-\sum_i c_{p,i} E_{p,i})$.  Set $A_p=\mu_p^*(\bbL)-\frac{1}{p}E_p$.

By the discussion after Lemma \ref{lem-Abd}, 
we known that $(Y_p, A_p)$ satisfies the first two conditions of Lemma \ref{lem-Abd}.  
So, by Lemma \ref{lem-Abd}, it suffices to show that there exists $C>0$ such that $A_X(E_{i,p})\le C p^{-1} c_{p,i}$ for any $i, p$. 

For any $\bQ$-line bundle $L$ on $S$ and with $h$ as above, define a compact convex set  following \cite[\S 2.b]{Nak04} (we identify line bundles with their Chern classes):
\begin{equation}
\Box(L, h)=\{(x, y)\in \bR_{\ge 0}^2; x+y\le 1 \text{ and } L+x(L_1-L_0)+y (L_2-L_0)\in {\rm PE}(S)\}.
\end{equation}

Then it is straight-forward to use \eqref{eq-nefcone} and \eqref{eq-c1Li} to get:
\begin{eqnarray*}
\Box(Q_p, h)=\left\{(x,y)\in \bR^2_{\ge 0}; \; \; x+\frac{5}{2\sqrt{6}}\frac{p-b}{p-a}y\le 1\right\}.
\end{eqnarray*}
Let $\vphi_{p, \min}$ be the metric of minimal singularity on $M_p$. Then it is known that $\cJ(\vphi_{p, \min})=\cJ(\|M_p\|)=\cJ_p$. 
For each fan $\sigma_i, i=0,1,2$, there exists an open set $U_i\cong S\times \bC^2$ which is an affine toric bundle over $S$.  
Applying the result in \cite[5.2]{Koi15}, $\cJ_p$ is trivial on $U_0, U_1$, and over $U_2$ we can choose the canonical affine coordinate $(z_1, z_2)$ on $\bC^2$ such that the multiplier ideal
is generated by monomials: 
\begin{equation}
\cJ_p(U_2)=\left\la z_1^{m_1} z_2^{m_2};\;\; (m_1+1, m_2+1)\in {\rm Int}((p-b)S_p)\cap \bZ^2\right\ra
\end{equation}
where
\begin{equation}
S_p=\left\{(x,y)\in \bR^2_{\ge 0}; \;\; x+\frac{y}{1-\frac{2\sqrt{6}(p-a)}{5(p-b)}}\ge 1 \right\},
\end{equation}
is generated by the exponents $(0, 1), (1,0), (0, 1-\frac{2\sqrt{6}(p-a)}{5(p-b)})$, which are the images of $(0,0)$, $(1,0)$, $(0,\frac{2\sqrt{6}(p-a)}{5(p-b)})$ under the linear map $(a,b)\mapsto (\la(a,b)-m_{\sigma_3}, v_i\ra)_{i=1,2}=(a, 1-a-b)$
with $m_{\sigma_3}=(0,1), v_1=(1,0), v_2=(-1, -1)$ (see \cite[Definition 4.1]{Koi15}).
In particular, the multiplier ideal sheaf is co-supported on $\bP(L_2)\subset \bP(L_0\oplus L_1\oplus L_2)$. Moreover, as pointed out in \cite{Koi15}, by using \cite[Theorem 2.10]{Nak04}, we know that the line bundle $\bbL$ does not admit birational Zariski decomposition.

Equivalently, we have:
\begin{equation}
\cJ_p(U_2)=\left\la z_1^{m_1}z_2^{m_2};\;\; (m_1, m_2)\in \bZ_{\ge 0}, \frac{m_1}{\alpha_p}+\frac{m_2}{\beta_p}> 1 \right\ra.
\end{equation}
where 
$
\alpha_p=\frac{p-b}{d_p}, \quad \beta_p=\frac{(1-\frac{2\sqrt{6}}{5})p+\frac{2\sqrt{6}}{5}a-b}{d_p}, \quad d_p=1-\frac{1}{p-b}-\frac{1}{(1-\frac{2\sqrt{6}}{5})p+\frac{2\sqrt{6}}{5}a-b}$. 
We only need to know that there exists $C=\frac{4.9-2\sqrt{6}}{5}>0$ such that  $\alpha_p\ge C p, \beta_p\ge C p$ for $p\gg 1$ since $d_p=1+O(p^{-1})$.

Because the multiplier ideal is monomial, we can use the result about Rees valuations of monomial ideals to see that the blow-up of $\cJ_p$ corresponds to the sides of the Newton-polygons of $\cJ_p$ (see \cite[15.4]{HS06}). Indeed, such blowup also corresponds to a subdivision of the cone $\sigma_2$. 

Now denote the sides of the Newton polygon be given by $\overline{P_{i-1} P_{i}}, 1\le i\le r$ with $P_i=(x_i, y_i)$. Then it is easy to see that
$P_0=(0, \lfloor \beta\rfloor +1)$ and $P_{r}=(\lfloor \alpha\rfloor+1, 0)$. Note that $(a_i, b_i)=(y_{i-1}-y_{i}, x_{i}-x_{i-1})\in \bR_{>0}^2$ is a normal vector of $\overline{P_{i-1}P_{i}}$. The monomial valuation $\ord_{E_i}$ that corresponds to the side $\overline{P_{i}P_{i+1}}$ can be chosen to be given by the weighted blowup with weights $(a_i, b_i)$. Set $\tau_i=b_i/a_i>0$. It is easy to see that:
\begin{eqnarray*}
w(E_i):=\frac{A(E_i)}{\ord_{E_i}(\cJ_p)}&=&\frac{a_i+b_i}{a_i x_{i-1}+b_i y_{i-1}}=\frac{1+\tau_i}{x_{i-1}+y_{i-1} \tau_i}\\
&=&\frac{a_i+b_i}{a_i x_i+b_i y_i}=\frac{1+\tau_i}{x_{i}+y_{i}\tau_i}.
\end{eqnarray*}
As a consequence, we have:
\begin{equation}
w(E_{i})-w(E_{i+1})=\frac{1+\tau_i}{x_i+y_i\tau_i}-\frac{1+\tau_{i+1}}{x_i+y_i\tau_{i+1}}=\frac{(y_i-x_i)(\tau_{i+1}-\tau_i)}{(x_i+y_i\tau_i)(x_i+y_i\tau_{i+1})}.
\end{equation}
From this identity, we easily see that $\max\{w(E_i); 1\le i\le r\}=\max\{w(E_1), w(E_r)\}$. 
Now note that $\tau_1^{-1}$ is at most the absolute value of the slope of the line $\overline{P_0 P'}$ where $P'$ is the point $(1, -\frac{\beta}{\alpha}+\beta)$ one the line connecting $(\alpha, 0)$ and $(0, \beta)$, which gives the inequality:
\begin{eqnarray*}
w(E_1)&=&\frac{1+\tau_1}{(\lfloor \beta\rfloor+1) \tau_1}=\frac{1}{\lfloor \beta\rfloor+1}+\frac{1}{(\lfloor \beta\rfloor+1) \tau_1}\le \frac{1}{\beta}+
\frac{1}{\lfloor \beta\rfloor+1}(\lfloor \beta\rfloor+1-\beta+\frac{\beta}{\alpha})\\
&=&\frac{1}{\beta}+\frac{\lfloor \beta\rfloor+1-\beta}{\lfloor\beta\rfloor+1}+\frac{\beta}{\lfloor \beta\rfloor+1}\frac{1}{\alpha}=O(p^{-1}).
\end{eqnarray*}
By the same argument (or just by symmetry), we also get $w(E_r)=O(p^{-1})$. According to the previous discussion, the verification of 3rd condition in Lemma \ref{lem-Abd} is complete.
\begin{rem}
It is easy to see that the above arguments, which reduce the problem to the estimates for Rees valuations of monomial ideals, works for many more examples of Nakayama type. 
\end{rem}



\vskip 3mm

\noindent
Department of Mathematics, Purdue University, West Lafayette, IN, 47907-2067.

\noindent
{\it Current address:} Department of Mathematics, Rutgers University, Piscataway, NJ 08854-8019.

\noindent
{\it E-mail address:} chi.li@rutgers.edu

\vskip 2mm







\begin{thebibliography}{999999}




\bibitem
{BBJ18}
R. Berman, S. Boucksom, M. Jonsson, A variational approach to the Yau-Tian-Donaldson conjecture, arXiv:1509.04561v2.






\bibitem
{Bil99}
P. Billingsley: Convergence of Probability Measures, 2nd edition, John Wiley, New York, 1999.

\bibitem
{BLX19}
H. Blum, Y. Liu and C. Xu,
Openness of K-semistability for Fano varieties, arXiv:1907.02408.

\bibitem
{Bo18}
S. Boucksom. Variational and non-Archimedean aspects of Yau-Tian-Donaldson conjecture, arXiv:1805.03289.

\bibitem
{BDPP13}
S. Boucksom, J.-P. Demailly, M. P\v{a}un, and T. Peternell: The pseudo-effective cone of a compact K\"{a}hler manifold and varieties of negative Kodaira dimension, J. Algebraic Geom. {\bf 22} (2013), no. 2, 201-248.


\bibitem
{BFJ09}
S. Boucksom, C. Favre, and M. Jonsson: Differentiability of volumes of divisors and a problem of Teissier, J. Algebraic Geom. {\bf 18} (2009), no. 2, 279-308.


\bibitem
{BFJ15}
S. Boucksom, C. Favre, and M. Jonsson: Solution to a non-Archimedean Monge-Amp\`{e}re equation. J. Amer. Math. Soc. {\bf 28} (2015), 617-667.

\bibitem
{BFJ16}
S. Boucksom, C. Favre, and M. Jonsson: Singular semipositive metrics in non-Archimedean geometry. J. Algebraic Geom. {\bf 25} (2016), no.1, 77-139.

\bibitem
{BHJ17}
S. Boucksom, T. Hisamoto and M. Jonsson: Uniform K-stability, Duistermaat-Heckman measures and singularities of pairs, Ann. Inst. Fourier (Grenoble) 67 (2017), 87-139. 


\bibitem
{BHJ19}
S. Boucksom, T. Hisamoto and M. Jonsson. Uniform K-stability and asymptotics of energy functionals in K\"{a}hler geometry, J. Eur. Math. Soc. (JEMS) {\bf 21} (2019), no. 9, 2905-2944.

\bibitem
{BoJ18a}
S. Boucksom, M. Jonsson: Singular semipositive metrics on line bundles on varieties over trivially valued fields. arXiv:1801.08229.


\bibitem
{BoJ18b}
S. Boucksom, M. Jonsson: A non-Archimedean approach to K-stability. arXiv:1805.11160v1.











\bibitem
{CL06}
A. Chambert-Loir: Mesures et \'{e}quidistribution sur des espaces de Berkovich, J. reine angew. Math. {\bf 595} (2006), p.215-235.

\bibitem
{CC18}
X. Chen, J. Cheng:
On the constant scalar curvature K\"{a}hler metrics, general automorphism group. arXiv:1801.05907.



\bibitem
{CS93}
S.D. Cutkosky, and V. Srinivas: On a problem of Zariski on dimensions of linear systems. Ann. of Math. {\bf 137} (1993), 531-559.



\bibitem
{Del20}
T. Delcroix:
Uniform K-stability of polarized spherical varieties, arXiv:2009.06463.

\bibitem
{Del20b}
T. Delcroix:
The Yau-Tian-Donaldson conjecture for cohomogeneity one manifolds, arXiv:2011.07135.




\bibitem
{DEL00}
J.-P. Demailly, L. Ein, and R. Lazarsfeld: A subadditivity property of multiplier ideas, Michigan Math. J. {\bf 48} (2000), 137-156.




\bibitem
{Don02}
S. Donaldson: Scalar curvature and stability of toric varieties, J. Differential Geom. {\bf 62} (2002), no.2, 289-349.





\bibitem
{ELMNP09}
L. Ein, R. Lazarsfeld, M. Musta\c{t}\u{a}, M. Nakamaye, and M. Popa: Restricted volumes and base loci of linear series, Amer. J. Math. {\bf 131} (2009), no. 3, 607-651.





\bibitem
{Hol10}
A. Holschbach:
A Chebotarev-type density theorem for divisors on algebraic varieties, arXiv:1006.2340.

\bibitem
{HS06}
C. Huneke, I. Swanson. Integral closure of ideals, rings, and modules, volume 336 of London Mathematical Society Lecture Note Series, 
Cambridge University Press, Cambridge, 2006.


\bibitem
{Koi15}
T. Koike: Minimal singular metrics of a line bundle admitting no Zariski decomposition. Tohoku Math. J. {\bf 67} (2015), 297-321.








\bibitem
{Laz04a}
R. Lazarsfeld: Positivity in algebraic geometry I, classical setting: line bundles and linear series. Ergebnisse der Mathematik und ihrer
Grenzgebiete, vols. 48 (Springer, Berlin, 2004)

\bibitem
{Laz04}
R. Lazarsfeld: Positivity in algebraic geometry II, positivity for vector bundles, and multiplier ideals. Ergebnisse der Mathematik und ihrer
Grenzgebiete, vols. 49 (Springer, Berlin, 2004).



\bibitem
{Li20}
C. Li: Geodesic rays and stability in the cscK problem, accepted by Annales Scientifiques de l'ENS, arXiv:2001.01366.

\bibitem
{LX14}
C. Li and C. Xu. Special test configurations and K-stability of Fano varieties, Ann. of Math. (2) {\bf 180} (2014), no.1, 197-232.



\bibitem
{Mat13}
S. Matsumura: Restricted volumes and divisorial Zariski decomposition, Amer. J. Math. {\bf 135} (2013), 637-662.

\bibitem
{Oda13}
Y. Odaka, A generalization of the Ross-Thomas slope theory, Osaka J. Math. {\bf 50} (2013), 171-185.


\bibitem
{Nak04}
N. Nakayama: Zariski-decomposition and abundance. MSJ Memoirs, {\bf 14}. Mathematical Society of Japan, Tokyo, 2004.

























\bibitem
{Tia97}
G. Tian: K\"{a}hler-Einstein metrics with positive scalar curvature. Invent. Math. {\bf 130} (1997), 239-265.


\bibitem
{Tia17}
G. Tian: K-stability implies CM-stability, Geometry, analysis and probability, 245-261, 
Progr. Math., 310, Birkh\"{a}user, Springer, Cham, 2017.


\bibitem
{Tsu06}
H. Tsuji: Pluricanonical systems of projective varieties of general type, I. Osaka J. Math. {\bf 43} (2006), no. 4, 967-995.






\bibitem
{Wan12}
X. Wang, Heights and GIT weight, Math. Res. Lett. {\bf 19} (2012), 909-926.



\end{thebibliography}
\end{document}